\documentclass{amsart}
\usepackage{amsfonts,amsmath,amsthm,amscd,amssymb,latexsym}

\usepackage{xy}
\xyoption{all}

\theoremstyle{plain}
\newtheorem{theorem}{Theorem}[section]
\newtheorem{lemma}[theorem]{Lemma}
\newtheorem{proposition}[theorem]{Proposition}
\newtheorem{corollary}[theorem]{Corollary}
\newtheorem{claim}[theorem]{Claim}
\newtheorem{problem}[theorem]{Problem}
\theoremstyle{definition}

\newcommand{\F}{\mathcal {F}}
\newcommand{\Ra}{\Rightarrow}
\newcommand{\U}{\mathcal U}
\newcommand{\V}{\mathcal V}

\newcommand{\A}{\mathcal A}

\newcommand{\LL}{\mathcal L}

\newcommand{\w}{\omega}

\newcommand{\IN}{\mathbb N}

\newcommand{\IC}{\mathbb C}
\newcommand{\C}{\mathcal C}
\newcommand{\E}{\mathcal E}

\newcommand{\inv[1]}{\overset{\leftrightarrow}{#1}}

\newcommand{\la}{\langle}
\newcommand{\ra}{\rangle}

\title{Algebra in superextensions of inverse semigroups}
\author{Taras Banakh and Volodymyr Gavrylkiv}
\address[T.~Banakh]{Ivan Franko University of Lviv, Ukraine and 
Jan Kochanowski University, Kielce, Poland}
\email{t.o.banakh@gmail.com}
\address[V.~Gavrylkiv]{Vasyl Stefanyk Precarpathian National University,
Ivano-Frankivsk,
Ukraine}
\email{vgavrylkiv@yahoo.com}
\thanks{The first author has been partially financed by NCN means granted by decision DEC-2011/01/B/ST1/01439}

\begin{document}

\begin{abstract}
We find necessary and sufficient conditions on an (inverse)
semigroup $X$ under which its semigroups of maximal linked systems
$\lambda(X)$, filters $\varphi(X)$, linked upfamilies $N_2(X)$,
and upfamilies $\upsilon(X)$ are inverse.
\end{abstract}
\maketitle

\section{Introduction}

In this paper we  investigate  the algebraic structure of various
extensions of an inverse semigroup  $X$ and detect semigroups
whose extensions $\lambda(X)$, $N_2(X)$, $\varphi(X)$,
$\upsilon(X)$ are inverse semigroups.

The thorough study of extensions of semigroups was started in
\cite{G2} and continued in \cite{BG2}--\cite{BGN}. The
largest among these extensions is the semigroup $\upsilon(X)$ of
all upfamilies on $X$.

A family $\F$ of subsets of a set $X$ is called an {\em upfamily}
if each set $F\in\F$ is not empty and for each set $F\in\F$ any
subset $E\supset F$ of $X$ belongs to $\F$. The space of all
upfamilies on $X$ is denoted by $\upsilon(X)$. It is a closed
subspace of the double power-set $\mathcal P(\mathcal P(X))$
endowed with the compact Hausdorff topology of the Tychonoff
product $\{0,1\}^{\mathcal P(X)}$. Identifying each point $x\in X$
with the upfamily $\la x\ra=\{A\subset X:x\in A\}$, we can
identify $X$ with a subspace of $\upsilon(X)$. Because of that we
call $\upsilon(X)$ an extension of $X$.

The compact Hausdorff space $\upsilon(X)$ contains many other
important extensions of $X$ as closed subspaces. In particular, it
contains the spaces $N_2(X)$ of linked upfamilies, $\lambda(X)$ of
maximal linked upfamilies, $\varphi(X)$ of filters, and  $\beta(X)$ of
ultrafilters on $X$; see \cite{G1}. Let us recall that an upfamily
$\F\in\upsilon(X)$ is called
\begin{itemize}
\item {\em linked} if $A\cap B\ne \emptyset$ for any sets
$A,B\in\F$; \item {\em maximal linked} if $\F=\F'$ for any linked
upfamily $\F'\in\upsilon(X)$ that contains $\F$; \item a {\em
filter} if $A\cap B\in\F$ for any $A,B\in\F$; \item an {\em
ultrafilter} if  $\F=\F'$ for any filter $\F'\in\upsilon(X)$ that
contains $\F$.
\end{itemize}
The family $\beta(X)$ of all ultrafilters on $X$ is called the
{\em Stone-\v Cech extension} and the family $\lambda(X)$ of all
maximal linked upfamilies is called the {\em superextension} of
$X$, see \cite{vM} and \cite{Ve}. The arrows in the following diagram
denote the identity inclusions between various extensions of a set
$X$.
$$\xymatrix{
X\ar[r]&\beta(X)\ar[r]\ar[d]&\lambda(X)\ar[d]&\\
 &\varphi(X)\ar[r]&N_2(X)\ar[r]&\upsilon(X).}
$$
 Any map $f:X\to Y$ induces a continuous map $$\upsilon f:\upsilon(X)\to\upsilon(Y),\quad \upsilon f:\F\mapsto \{A\subset Y:f^{-1}(A)\in\F\},$$ such that $\upsilon f(\beta(X))\subset\beta(Y)$,
$\upsilon f(\lambda(X))\subset\lambda(Y)$, $\upsilon
f(\varphi(X))\subset\varphi(Y)$, and $\upsilon f(N_2(X))\subset
N_2(Y)$. If the map $f$ is injective, then  $\upsilon f$ is a
topological embedding, which allows us to identify the extensions
$\beta(X)$, $\lambda(X)$, $\varphi(X)$, $N_2(X)$, $\upsilon(X)$
with corresponding closed subspaces in $\beta(Y)$, $\lambda(Y)$,
$\varphi(Y)$, $N_2(Y)$, and $\upsilon(Y)$, respectively.

In \cite{G2} it was observed that any (associative) binary
operation $*:X\times X\to X$ can be extended  to an
(associative) binary operation
$*:\upsilon(X)\times\upsilon(X)\to\upsilon(X)$ defined by the
formula:
$$\mathcal A*\mathcal B=\big\la\bigcup_{a\in A}a*B_a:A\in\A,\;\;\{B_a\}_{a\in A}\subset\mathcal B\big\ra$$for upfamilies $\A,\mathcal B\in\upsilon(X)$.
Here for a family $\C$ of non-empty subsets of $X$ by
$$\la\C\ra=\{A\subset X:\exists C\in\C\mbox{ with }C\subset
A\}$$we denote the upfamily generated by the family $\C$.

According to \cite{G1}, for each semigroup $X$, $\upsilon(X)$ is a compact
Hausdorff right-topological semigroup containing the subspaces
$\beta(X)$, $\lambda(X)$, $\varphi(X)$, $N_2(X)$ as closed
subsemigroups. Algebraic and topological properties
of these semigroups have been studied in \cite{G2},
\cite{BG2}--\cite{BGN}. In particular, in \cite{BG10} we studied
properties of extensions of groups while \cite{BGs} was devoted to
extensions of semilattices. There are two important classes of
semigroups that include all groups and all semilattices. Those are
the classes of inverse and Clifford semigroups.

Let us recall that a semigroup $S$ is {\em inverse} if for any
element $x\in S$ there is a unique element $x^{-1}$ (called {\em
the inverse of} $x$) such that $xx^{-1}x=x$ and
$x^{-1}xx^{-1}=x^{-1}$. A semigroup $S$ is {\em regular} if each
element $x\in S$ is {\em regular} in the sense that $x\in xSx$. It
is known \cite[II.1.2]{Pet} that a semigroup $S$ is inverse if and
only if $S$ is regular and idempotents of $S$ commute. For a
semigroup $S$ by $E=\{x\in S:xx=x\}$ we denote the set of
idempotents of $S$. It follows that for an inverse semigroup, $E$
is a commutative subsemigroup of $S$ and hence $E$ is a {\em
maximal semilattice} in $S$. Let us recall that a {\em
semilattice} is a set endowed with an associative commutative
idempotent operation. A semigroup $S$ is called {\em linear} if
$xy\in\{x,y\}$ for any points $x,y\in X$. Each linear commutative
semigroup is a semilattice.

A semigroup $S$ is {\em Clifford} if it is a union of groups. An
inverse semigroup $S$ is Clifford if and only if $xx^{-1}=x^{-1}x$
for all $x\in X$. A semigroup $S$ is called {\em sub-Clifford} if
it is a union of cancellative semigroups. A semigroup $S$ is
sub-Clifford if and only if for any positive integer numbers $n<m$
and any $x\in S$ the equality $x^{n+1}=x^{m+1}$ implies $x^n=x^m$.
Each subsemigroup of a Clifford semigroup is sub-Clifford and each
finite sub-Clifford semigroup is Clifford. A commutative semigroup
is Clifford if and only if it is inverse if and only if it is
regular. A semigroup $X$ is called {\em Boolean} if $x^3=x$ for
all $x\in X$. Each Boolean semigroup is Clifford.

It is well-known that the class of inverse (Clifford) semigroups
includes all groups and all semilattices. Moreover, each inverse
Clifford semigroup $S$ decomposes into the union $S=\bigcup_{e\in
E}H_e$ of maximal subgroups $H_e=\{x\in S:xx^{-1}=e=x^{-1}x\}$
indexed by the idempotents, which commute with all elements of
$S$; see \cite[II.2]{Pet}.

The algebraic structure of extensions of groups was studied in
details in \cite{BG10}. Extensions of semilattices were
investigated in \cite{BGs}. In particular, in \cite{BGs} it was
shown  that for a semigroup $X$ the superextension $\lambda(X)$ is
a semilattice if and only if the semigroup $\upsilon(X)$ is a
semilattice if and only if $X$ is a finite linear semilattice.

In Theorems~\ref{t1l}--\ref{t1u} below we shall list all
semigroups $X$ whose extensions are (commutative) inverse
semigroups.

For a natural number $n$ by $C_n=\{z\in\IC:z^n=1\}$ we denote the
cyclic group of order $n$ and by $L_n$ the linear semilattice
$\{0,\dots,n-1\}$ of order $n$, endowed with the operation of
minimum. In particular, $L_0$ is an empty semigroup.

For two semigroups $(X,*)$ and $(Y,\star)$ by $X\sqcup Y$
we denote the disjoint union of these semigroups endowed with the
semigroup operation
$$
x\circ y=\begin{cases}
x* y&\mbox{if $x,y\in X$,}\\
x&\mbox{if $x\in X$ and $y\in Y$,}\\
y&\mbox{if $x\in Y$ and $y\in X$,}\\
x\star y&\mbox{if $x,y\in Y$}.
\end{cases}
$$
The semigroup $X\sqcup Y$ will be called the {\em disjoint ordered union} of the semigroups $X$ and $Y$. Observe that the operation of disjoint ordered union is associative in the sense that $(X\sqcup Y)\sqcup Z=X\sqcup (Y\sqcup Z)$ for any semigroups $X,Y,Z$.

We say that a subsemigroup $X$ of a semigroup $Y$ is {\em regular
in $Y$} if each element $x\in X$ is regular in $Y$.

\begin{theorem}\label{t1l} For a semigroup $X$ and its superextension $\lambda(X)$ the following conditions are
equivalent:
\begin{enumerate}
\item[{\rm(1)}] $\lambda(X)$ is a commutative Clifford semigroup;
\item[{\rm(2)}] $\lambda(X)$ is an inverse semigroup;
\item[{\rm(3)}] the idempotents of the
semigroup $\lambda(X)$ commute and $\lambda(X)$ is sub-Clifford or
regular in $N_2(X)$;
\item[{\rm(4)}] $X$ is a finite commutative Clifford
semigroup, isomorphic to one of the following semigroups:
 $C_2$, $C_3$, $C_4$, $C_2\times C_2$, $L_2\times C_2$, $L_1\sqcup C_2$, $L_n$,
or $C_2\sqcup L_n$ for some $n\in\w$.
\end{enumerate}
\end{theorem}

\begin{theorem}\label{t1f} For a semigroup $X$ and its semigroup of filters $\varphi(X)$ the following conditions are
equivalent:
\begin{enumerate}
\item[{\rm(1)}] $\varphi(X)$ is a commutative Clifford semigroup;
\item[{\rm(2)}] $\varphi(X)$ is an inverse semigroup;
\item[{\rm(3)}] the idempotents of the
semigroup $\varphi(X)$ commute and $\varphi(X)$ is sub-Clifford or
regular in $N_2(X)$;
\item[{\rm(4)}] $X$ is isomorphic to one of the
semigroups: $C_2$, $L_n$ or $L_n\sqcup C_2$ for some $n\in\w$.
\end{enumerate}
\end{theorem}

\begin{theorem}\label{t1n} For a semigroup $X$ and its semigroup of linked upfamilies  $N_2(X)$ the following conditions are
equivalent:
\begin{enumerate}
\item[{\rm(1)}] $N_2(X)$ is a commutative Clifford semigroup;
\item[{\rm(2)}] $N_2(X)$ is an inverse semigroup;
\item[{\rm(3)}] the idempotents of the semigroup
$N_2(X)$ commute and $N_2(X)$ is sub-Clifford or regular;
\item[{\rm(4)}] $X$ is isomorphic to $C_2$ or  $L_n$ for some $n\in\w$.
\end{enumerate}
\end{theorem}

\begin{theorem}\label{t1u} For a semigroup $X$ and its semigroup of upfamilies $\upsilon(X)$ the following conditions are equivalent:
\begin{enumerate}
\item[{\rm(1)}] $\upsilon(X)$ is a finite semilattice;
\item[{\rm(2)}] $\upsilon(X)$ is an inverse semigroup;
\item[{\rm(3)}] the idempotents of the semigroup
$\upsilon(X)$ commute and $\upsilon(X)$ is sub-Clifford or
regular;
\item[{\rm(4)}] $X$ is a finite linear semilattice, isomorphic to $L_n$ for some $n\in\w$.
\end{enumerate}
\end{theorem}

Surprisingly, the following problem remains open.

\begin{problem} Characterize semigroups $X$ whose Stone-\v Cech extension $\beta(X)$ is an inverse semigroup. {\rm (Such semigroups have finite linear and finite cyclic subsemigroups; see Proposition~\ref{p2.1}.)}
\end{problem}

Theorems~\ref{t1l}, \ref{t1f}, \ref{t1n}, and \ref{t1u} will be proved
in Sections~\ref{s:t1l}, \ref{s:t1f}, \ref{s:t1n}, and \ref{s:t1u}, respectively.

\section{Commutativity in the Stone-\v Cech extension}

In this section we establish some properties of semigroups whose
Stone-\v Cech extension has commuting idempotents. Let us recall
that a semigroup $S$ is {\em cyclic} if $S=\{x^n:n\in\IN\}$
for some element $x\in S$, called the {\em generator} of $S$.

\begin{proposition}\label{p2.1} If for a semigroup $X$ all idempotents of the Stone-\v Cech extension $\beta(X)$ commute, then all cyclic subsemigroups and all linear subsemigroups of $X$ are finite.
\end{proposition}

\begin{proof} First we show that each element $x\in X$ generates a finite cyclic subsemigroup $\{x^n\}_{n\in\IN}$. If $\{x^n\}_{n\in\IN}$ is infinite, then it is isomorphic to the semigroup $(\IN,+)$. Then the Stone-\v Cech extension $\beta(X)$ contains a subsemigroup isomorphic to the Stone-\v Cech extension $\beta(\IN)$ of the semigroup $(\IN,+)$. By Theorem 6.9 of \cite{HS}, the semigroup $\beta(\IN)$ contains $2^{\mathfrak c}$ non-commuting idempotents and so does the semigroup $\beta(X)$ which is forbidden by our assumption. So, the cyclic subsemigroup $\{x^n\}_{n\in\IN}$ is finite.

Next, assume that $X$ contains an infinite linear subsemigroup
$L$. Then $xy\in\{x,y\}$ for any elements $x,y\in L$. Choose any
injective sequence  $\{x_n\}_{n\in\w}$ in $L$ and define a
2-coloring $\chi:[\w]^2\to \{0,1\}$ of the set
$[\w]^2=\{(n,m)\in\w^2:n<m\}$ letting
$$\chi(n,m)=
\begin{cases}
0&\mbox{if $x_nx_m=x_n$}\\
1&\mbox{if $x_nx_m=x_m$}.
\end{cases}
$$ By Ramsey's Theorem \cite{Ramsey} (see also \cite[Theorem 5]{GRS}), there is an infinite subset $\Omega\subset\w$ and a color $k\in\{0,1\}$ such that $\chi(n,m)=k$ for any pair $(n,m)\in[\w]^2\cap\Omega^2$. Consider the infinite linear subsemigroup $Z=\{x_n\}_{n\in\Omega}$ of $X$. By Theorem 1.1 of \cite{BGs}, each element of the semigroup $\beta(Z)$ is an idempotent. We claim that any two distinct free ultrafilters $\U,\V\in\beta(Z)$ do not commute (which is forbidden by our assumption).
If the color $k=0$, then $x_nx_m=x_n$ for any numbers $n<m$ in
$\Omega$, which implies that $\U*\V=\U\ne \V=\V*\U$. If $k=1$,
then $x_nx_m=x_m$ for any numbers $n<m$ in $\Omega$ and then
$\U*\V=\V\ne\U=\V*\U$. \end{proof}

\section{The regularity of extensions of semigroups}

In this section we shall prove some results related to the
regularity of semigroups. Let us recall that an element $x\in S$ is {\em
regular} in a semigroup $S$ if $x\in xSx$.

\begin{proposition}\label{p3.1} Let $X$ be a semigroup. An element $x\in X$ is regular in $X$ if and only if the ultrafilter $\la x\ra$ is regular in the semigroup $\upsilon(X)$.
\end{proposition}

\begin{proof} The ``if'' part is trivial. To prove the ``only if'' part, assume that $\la x\ra$ is regular in $\upsilon(X)$ and find an upfamily $\F\in\upsilon(X)$ such that $\la x\ra=\la x\ra*\F*\la x\ra$. Then for some set $F\in\F$ we get $x\in xFx\subset xSx$, which means that $x$ is regular in $X$.
\end{proof}

\begin{corollary} A semigroup $X$ is inverse if and only if $X$ lies in some inverse semigroup $S\subset\upsilon(X)$.
\end{corollary}

\begin{proof} The ``only if'' part is trivial (just take $S=X$). To prove the ``if'' part, assume that a semigroup $X$ lies in some inverse subsemigroup $S\subset\upsilon(X)$. The inverse semigroup $S$ is regular and has commuting idempotents. Then the idempotents of the subsemigroup $X\subset S$ also commute. Each element $x\in X\subset S$ is regular in $S$ and hence is regular in $X$ by Proposition~\ref{p3.1}. Then the semigroup $X$ is inverse, being a regular semigroup with commuting idempotents; see \cite[II.1.2]{Pet}.
\end{proof}

Let us recall that a non-empty subset $I$ of a semigroup $X$ is
called an {\em ideal} in $X$ if $XI\cup IX\subset I$.

\begin{lemma}\label{l3.3} Let $X$ be a semigroup and $Z\subset X$ be a subsemigroup whose complement $X\setminus Z$ is an ideal in $X$. If for two upfamilies $\A\in\upsilon(Z)\subset\upsilon(X)$ and $\mathcal B\in\upsilon(X)$ we get $\A=\A*\mathcal B*\A$, then $\A=\A*{\mathcal B}_Z*\A$ for the upfamily $\mathcal B_Z=\{B\in\mathcal B:B\subset Z\}\in\upsilon(Z)$.
\end{lemma}

\begin{proof} It is clear that $\A*\mathcal B_Z*\A\subset \A*\mathcal B*\A=\A$. To prove the reverse inclusion, take any set $A\in\A$. It follows from $\A\in\upsilon(Z)\subset\upsilon(X)$ that $A\cap Z\in\A\subset(\A*\mathcal B)*\A$. So, we can find a set $C\in \A*\mathcal B$ and a family $\{A_c\}_{c\in C}\subset\A$ such that $\bigcup_{c\in C}c*A_c\subset A\cap Z$. For every $c\in C$ the inclusion $c*A_c\subset A\cap Z\subset Z$ implies $c\in Z$ (because $X\setminus Z$ is an ideal in $X$). So, $C\subset Z$. Since $C\in\A*\mathcal B$, there is a set $A\in\A$ and a family $\{B_a\}_{a\in A}\subset\mathcal B$ such that $\bigcup_{a\in A}a*B_a\subset C$. Since $X\setminus Z$ is an ideal in $X$, for every $a\in A$ the inclusion $a*B_a\subset C\subset Z$ implies $B_a\subset Z$ which means that $\{B_a\}_{a\in A}\subset\mathcal B_Z$ and hence $\A\subset\A*\mathcal B_Z*\A$.
\end{proof}

\begin{corollary}\label{c3.4} Let $X$ be a semigroup and $S\in\{\beta(X),\lambda(X),\varphi(X),  N_2(X), \upsilon(X)\}$ be one of its extensions. If the semigroup $S$ is regular, then for any subsemigroup $Z\subset X$ whose complement $X\setminus Z$ is an ideal in $X$ the semigroup $S\cap \upsilon(Z)$ is regular.
\end{corollary}

\begin{proof} Fix any upfamily $\A\in S\cap\upsilon(Z)$ and by the regularity of the semigroup $S$, find an upfamily $\mathcal B\in S$ such that $\A=\A*\mathcal B*\A$. By Lemma~\ref{l3.3}, $\A=\A*\mathcal B_Z*\A$ for the upfamily $\mathcal B_Z=\{B\in\mathcal B:B\subset Z\}\in\upsilon(Z)$. If $S\in\{\varphi(X),N_2(X),\upsilon(X)\}$, then $\mathcal B_Z\in S\cap \upsilon(Z)$ and hence $\A$ is regular in $S\cap\upsilon(Z)$.

If $S=\beta(X)$, then $\mathcal B_Z$ is a filter on $Z$ and we can
enlarge it to an ultrafilter $\tilde{\mathcal B}_Z\in\beta(Z)$.
Then $\A=\A*\mathcal B_Z*\A\subset\A*\tilde{\mathcal B}_Z*\A$
implies that $\A=\A*\tilde{\mathcal B}_Z*\A$ by the maximality of
the ultrafilter $\A$. So, $\A$ in regular in the semigroup
$\beta(Z)$. By analogy we can consider the case $S=\lambda(X)$.
\end{proof}

\section{The extensions of the exceptional semigroups from\newline Theorem~\ref{t1l}}

In this section we describe the structure of the extensions of the
exceptional semigroups from Theorem~\ref{t1l}(4).

We start with studying the superextensions of these semigroups.
First note that for each set $X$ of cardinality $1\le |X|\le 2$
the superextension $\lambda(X)$ coincides with $\beta(X)=X$. If a
set $X$ has cardinality $|X|=3$, then
$\lambda(X)=X\cup\{\triangle\}$ where $\triangle=\{A\subset
X:|A|\ge2\}$. For a set $X$ of cardinality $|X|=4$ the
superextension $\lambda(X)=\{x,\triangle_x,\square_x:x\in X\}$
consists of 12 elements, where
$$
\begin{aligned}
\triangle_x&=\{A\subset X:|A\setminus\{x\}|\ge 2\}\mbox{ \ and \ }\\
\square_x&=(X\setminus\{x\})\cup \{A\subset X:x\in A,\;|A|\ge 2\}\mbox{ \ for $x\in X$}.
\end{aligned}
$$
Given two semigroups $X,Y$ we shall write $X\cong Y$ if these
semigroups are isomorphic.

\begin{proposition}\label{p4.1} For finite exceptional semigroups we have the following isomorphisms:
\begin{enumerate}
\item[{\rm(1)}] $\lambda(C_2)=C_2$.
\item[{\rm(2)}] $\lambda(C_3)\cong L_1\sqcup C_3$.
\item[{\rm(3)}] $\lambda(C_4)\cong(C_2\sqcup L_1)\times C_4$.
\item[{\rm(4)}] $\lambda(C_2\times C_2)\cong(C_2\sqcup L_1)\times C_2\times C_2$.
\item[{\rm(5)}] $\lambda(L_1\sqcup C_2)\cong L_1\sqcup L_1\sqcup C_2$.
\item[{\rm(6)}] $\lambda(L_2\times C_2)\cong\big(L_1\sqcup (L_2\times L_2)\sqcup
L_1\big)\times C_2$.
\end{enumerate}
\end{proposition}

\begin{proof}
1--4. The first four statements were proved in \cite[\S6]{BGN}.

5. For the semigroup $X=L_1\sqcup C_2=\{0,1,-1\}$ the
superextension $\lambda(X)=\{0,\triangle,1,-1\}$ has the structure
of the ordered union $\{0\}\sqcup \{\triangle\}\sqcup\{1,-1\}$,
which is isomorphic to the semigroup $L_1\sqcup L_1\sqcup C_2$.

6. The semigroup $L_2\times C_2=\{0,1\}\times\{-1,1\}$ has two
idempotents $e=(0,1)$ and $f=(1,1)$ and two elements $a=(0,-1)$
and $b=(1,-1)$ of order 2 such that $a^2=e$ and $b^2=f$. The
superextension $$\lambda(L_2\times
C_2)=\{x,\triangle_x,\square_x:x\in L_2\times C_2\}$$ has the
6-element set of idempotents
$E=\{e,\square_e,\triangle_a,\triangle_b,\square_f,f\}$,
isomorphic to the semilattice $L_1\sqcup(L_2\times L_2)\sqcup
L_1$. The semigroup $\lambda(L_2\times C_2)$ is isomorphic to the
product $E\times C_2$ under the isomorphism $h:E\times C_2\to
\lambda(L_2\times C_2)$ defined by
$$h:(x,g)\mapsto\begin{cases}
x&\mbox{if $g=1$},\\
xb&\mbox{if $g=-1$}.
\end{cases}
$$
\end{proof}

The following proposition was proved in \cite[3.1]{BGs}.

\begin{proposition}\label{p4.2} For every $n\in\IN$ the semigroup $\upsilon(L_n)$ is a finite semilattice. Consequently, the semigroups $\lambda(L_n)$, $\varphi(L_n)$, $N_2(L_n)$ also are finite semilattices.
\end{proposition}

We recall that a semigroup $X$ is called {\em Boolean} if $x=x^3$
for all $x\in X$. It is clear that each Boolean semigroup is
Clifford and each commutative Boolean semigroup is inverse.

\begin{proposition}\label{p4.3} For every $n\in\IN$ and the semigroup $X=C_2\sqcup L_n$ the superextension $\lambda(X)$ is a finite commutative Boolean semigroup whose maximal semilattice $E\big(\lambda(X)\big)$ coincides with the set $\lambda(X)\setminus\{a\}$ where $a$ is the unique element generating the subgroup $C_2$ of $X=C_2\sqcup L_n$. Moreover,  $ea=a$ for any idempotent $e$ of $\lambda(X)$.
\end{proposition}

\begin{proof} Observe that $X\setminus\{a\}$ is a linear semilattice such that
$xa=a$ for all $x\in X\setminus\{a\}$. We identify the point $a$
with the principal ultrafilter $\la a\ra$ generated by $a$. For
the convenience of the reader we divide the proof of
Proposition~\ref{p4.3} into a series of claims.

\begin{claim} For each $\F\in \lambda(X)\setminus\{a\}$ we get $a*\F=\F*a=\la a\ra$.
\end{claim}

\begin{proof} Since $\F\ne \la a\ra$, there is a set $F\in\F$ with $a\notin F$.
Then $a*F=F*a=\{a\}$, which implies $a*\F=a*\F=\la a\ra$.
\end{proof}

\begin{claim} Each element $\F\in \lambda(X)\setminus\{a\}$ is an idempotent.
\end{claim}

\begin{proof} Since the upfamilies $\F$ and $\F*\F$ are maximal linked, it suffices to check that $\F\subset\F*\F$. Fix any set $F\in\F$ and consider two cases. If $a\notin F$, then $F=F*F\in\F*\F$.
So, assume that $a\in F$. Since $\F\ne\la a\ra$, there is a
non-empty set $F_a\in \F$ that does not contain the point $a$.
Then $a*F_a=\{a\}\subset F$. For each $x\in F\setminus\{a\}$, let
$F_x=F$ and observe that $x*F\subset \{x\}\cup F\subset F$. Then
$\bigcup_{x\in F}x*F_x\subset\{a\}\cup F=F$ and hence $F\in\F*\F$.
\end{proof}

\begin{claim} $\U*\V=\V*\U$ for any maximal linked systems $\U,\mathcal V\in \lambda(X)$.
\end{claim}

\begin{proof}

The equality $\U*\V=\V*\U$ is trivial if $\U$ or $\V$ belongs to
$\beta(X)=X$. So, we assume that the maximal linked systems
$\U,\V\notin X$ are not ultrafilters.

First we prove that $\U*\V\subset\V*\U$. Fix any set $W\in\U*\V$.
Without loss of generality, it is of the basic form
$W=\bigcup_{u\in U}u*V_u$ for some set $U\in\U$ and a family
$\{V_u\}_{u\in U}\subset\V$. Since $X\setminus\{a\}$ is a linear
semilattice, (the proof of Theorem 2.5) \cite{BGs} guarantees
that:

$$(U\setminus\{a\})*(V_u\setminus\{a\})\subset W\mbox{ \ \ for some point \ \ }u\in U\setminus\{a\}.$$

We consider three cases.
\smallskip

1) $a\notin V_{u}$ and $a\notin U$. Then $\V*\U\ni
V_u*U=U*V_{u}=(U\setminus\{a\})*(V_{u}\setminus\{a\})\subset W$
and hence $W\in\V*\U$.
\smallskip

2) $a\notin V_{u}$ and $a\in U$. Then $a*V_u=\{a\}$. Since
$\V\ne\la a\ra$, the set $V_a\setminus\{a\}$ is not empty and
hence contains some idempotent. Then $aV_u=\{a\}\subset
a*V_a\subset W$ and $$ \V*\U\ni
V_u*U=U*V_{u}=(a*V_u)\cup(U\setminus\{a\})*V_u =$$ $$\{a\}\cup
(U\setminus\{a\})*(V_u\setminus\{a\})\subset\{a\}\cup W\subset W$$
and again $W\in\V*\U$.
\smallskip

3) $a\in V_u$. In this case $a\in u*V_u\subset W$. It follows from
$\U\ne\la a\ra$ that $a\notin U_a$ for some set $U_a\in \U$. Let
$U_v=U$ for all $v\in V_u\setminus\{a\}$ and observe that
$$\V*\U\ni\bigcup_{v\in V_u}v*U_v=a*U_a\cup\big((V_u\setminus\{a\})*U)=$$ $$\{a\}\cup\big((V_u\setminus\{a\})*(U\setminus\{a\})\big)\cup \big((V_u\setminus\{a\})*a\big)\subset\{a\}\cup W\cup\{a\}\subset W$$and hence $W\in\V*\U$.
Therefore, $\U*\V\subset\V*\U$.

The inclusion $\V*\U\subset\U*\V$ can be proved by analogy.
\end{proof}
\end{proof}

Next, we study the structure of the space of filters $\varphi(X)$
of the finite exceptional groups from Theorem~\ref{t1f}.

\begin{proposition}\label{p4.7}
\begin{enumerate}
\item[{\rm(1)}] $\varphi(C_2)=N_2(C_2)\cong L_1\sqcup C_2$;
\item[{\rm(2)}] $\varphi(L_1\sqcup C_2)$ is a commutative Boolean semigroup
isomorphic to the subsemigroup $$\{(e,x)\in(L_1\sqcup(L_2\times
L_2))\times C_2:e\in L_1\sqcup\{(0,0),(0,1)\}\Ra (x=1)\}$$ of the
commutative Boolean semigroup $\big(L_1\sqcup(L_2\times
L_2)\big)\times C_2$.
\end{enumerate}
\end{proposition}

\begin{proof} 1. The semigroup $\varphi(C_2)$ contains two ultrafilters and one filter $\mathcal Z=\la C_2\ra$ generated by the set $C_2$. The filter $\mathcal Z$ is the zero of the semigroup $\varphi(C_2)$ and hence $\varphi(C_2)$ is isomorphic to $\{\mathcal Z\}\sqcup C_2$.
\smallskip

2. For the semigroup $X=L_1\sqcup C_2=\{0,1,-1\}$ the semigroup
$\varphi(X)$ contains 7 filters generated by all non-empty subsets
of $X$. So, we can identify filters with their generating sets.
Among these 7 filters there are 5 idempotents: $\{0\}$,
$\{0,1,-1\}$, $\{0,1\}$, $\{1,-1\}$, and $\{1\}$ which form a
semilattice  $E$
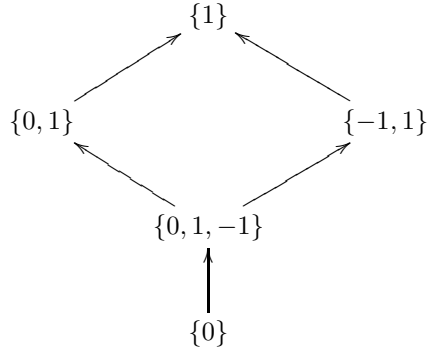
$$\xymatrix{&\{1\}&\\
\{0,1\}\ar[ru]&& \{-1,1\}\ar[lu]\\
&\{0,1,-1\}\ar[lu]\ar[ru]\\
&\{0\}\ar[u] }$$isomorphic to $L_1\sqcup(L_2\times L_2)$. Two
filters $\{-1\}$ and $\{0,-1\}$ generate 2-element subgroups with
idempotents $\{1\}$ and $\{0,1\}$, respectively. Since
$\{-1\}*\{0,1\}=\{0,-1\}$, the semigroup $\lambda(X)$ is
isomorphic to the subsemigroup $\{(e,x)\in E\times
C_2:e\in\big\{\{0\},\{0,1,-1\},\{1,-1\}\big\}\Ra (x=1)\}$ of the
commutative Boolean semigroup $E\times C_2$.
\end{proof}

Now we consider the the semigroups $\varphi(L_n)$
and $\varphi(L_n\sqcup C_2)$. We shall show that the latter
semigroup has the structure of the reduced product of a
semilattice and a group.

Let $X,Y$ be two semigroups and $I$ be an ideal in $X$. The {\em
reduced product} $X\times_I Y$ is the set $I\cup \big((X\setminus
I)\times Y\big)$ endowed with the semigroup operation
$$a*b=\begin{cases}
p_X(a)*p_X(b)&\mbox{if $p_X(a)*p_X(b)\in I$},\\
(p_X(a)*p_X(b),p_Y(a)*p_Y(b))&\mbox{if $p_X(a)*p_X(b)\notin I$}.
\end{cases}$$
Here by $p_X:X\times_I Y\to X$ and $p_Y:(X\setminus I)\times Y\to
Y$ we denote the natural projections. Let us recall that by
Proposition~\ref{p4.7}(1), the semigroup $\varphi(C_2)$ is
isomorphic to the commutative Boolean semigroup $L_1\sqcup C_2$.

\begin{proposition}\label{p4.8} For every $n\in\IN$ the semigroup
\begin{enumerate}
\item[{\rm(1)}] $\varphi(L_n)$ is a finite semilattice, and
\item[{\rm(2)}] $\varphi(L_n\sqcup C_2)$ is a commutative Boolean semigroup
isomorphic to the reduced product
$\varphi(L_{n+1})\times_{\varphi(L_n)}\varphi(C_2)$.
\end{enumerate}
\end{proposition}

\begin{proof} By Proposition~\ref{p4.2}, the semigroup $\varphi(L_n)$ is a finite semilattice.

Now consider the semigroup $X=L_n\sqcup C_2$. Since $X$ is finite
we can identify the semigroup $\varphi(X)$ with the commutative
semigroup of all non-empty subsets of $X$. Let $a$ be the
generator of the cyclic group $C_2$ and $e=a^2$ be its idempotent.
The idempotent semilattice $E=L_n\sqcup\{e\}$ of $X$ is isomorphic
to the linear semilattice $L_{n+1}$. So, we shall identify $E$
with $L_{n+1}$. Observe that $\varphi(X)\setminus
\varphi(L_n)=\{F\subset X:F\cap C_2\ne\emptyset\}$ and the map
$h:\varphi(L_{n+1})\times_{\varphi(L_n)}
\varphi(C_2)\to\varphi(X)$ defined by
$$h(A)=\begin{cases}
A&\mbox{if $A\subset\varphi(L_{n})$}\\
(A\setminus C_2)\cup B&\mbox{if
$(A,B)\in(\varphi(L_{n+1})\setminus
\varphi(L_n))\times\varphi(C_2)$}
\end{cases}
$$
is a required isomorphism between the semigroups $\varphi(X)$ and
$\varphi(L_{n+1})\times_{\varphi(L_n)} \varphi(C_2)$.
\end{proof}

\section{Proof of Theorem~\ref{t1l}}\label{s:t1l}

Given a semigroup $X$, we need to prove the equivalence of the
following statements:
\begin{enumerate}
\item[(1)] $\lambda(X)$ is a commutative Clifford semigroup;
\item[(2)] $\lambda(X)$ is an inverse semigroup;
\item[(3)] the idempotents of $\lambda(X)$ commute and $\lambda(X)$ is sub-Clifford or regular
in $N_2(X)$;
\item[(4)] $X$ is a finite commutative inverse semigroup,
isomorphic to one of the following semigroups:
$C_2$, $C_3$, $C_4$, $C_2\times C_2$, $L_2\times C_2$, $L_1\sqcup C_2$, $L_n$,
or $C_2\sqcup L_n$ for some $n\in\w$.
\end{enumerate}

We shall prove the implications $(4)\Ra(1)\Ra(2)\Ra(3)\Ra(4)$.

The implication $(4)\Ra(1)$ follows from
Propositions~\ref{p4.1}---\ref{p4.3} while $(1)\Ra(2)\Ra(3)$ are
trivial or well-known; see \cite[II.1.2]{Pet}.

To prove that $(3)\Ra(4)$, assume that the idempotents of the
semigroup $\lambda(X)$ commute and $\lambda(X)$ is sub-Clifford or
regular in $N_2(X)$. Then the idempotents of the semigroup $X$
also commute and hence the set $E=\{e\in X:ee=e\}$ of idempotents
of $X$ is a semilattice. For the convenience of the reader we
divide the further proof into a series of claims.

\begin{claim}\label{cl6.1} The semigroup $\lambda(X)$ is sub-Clifford or regular.
\end{claim}

\begin{proof} If $\lambda(X)$ is not sub-Clifford, then it is regular in the semigroup $N_2(X)$ according to our assumption. We claim that $\lambda(X)$ is regular. Given any maximal linked system $\A\in\lambda(X)$, use the regularity of $\lambda(X)$ in $N_2(X)$ to find a linked upfamily $\mathcal B\in N_2(X)$ such that $\A=\A*\mathcal B*\A$. Enlarge $\mathcal B$ to a maximal linked upfamily $\tilde{\mathcal B}\in\lambda(X)$. Then $\A=\A*\mathcal B*\A\subset \A*\tilde{\mathcal B}*\A$ implies that $\A=\A*\tilde{\mathcal B}*\A$ by the maximality of the linked family $\A$. Therefore $\A$ is regular in $\lambda(X)$.
\end{proof}

\begin{claim}\label{cl6.2} The semigroup $X$ is inverse.
\end{claim}

\begin{proof} Since the idempotents of the semigroup $X$ commute, it suffices to check that $X$ is regular. By our assumption, the semigroup $X$ is regular or sub-Clifford. If $\lambda(X)$ is regular, then $X$ is regular by Proposition~\ref{p3.1}. Now assume that $\lambda(X)$ is sub-Clifford. Then so is the semigroup $X$. Since the idempotents of the semigroup $\beta(X)\subset\lambda(X)$ commute, by Proposition~\ref{p2.1}, each cyclic subsemigroup $\{x^n\}_{n\in\IN}$ of $S$ is finite and hence is a group by the sub-Clifford property of $X$. Then the semigroup $X$ is Clifford and hence regular.
\end{proof}

\begin{claim}\label{cl6.3} The semilattice $E\subset X$ is linear and finite.
\end{claim}

\begin{proof} Assuming that $E$ is not linear, we can find two idempotents $x,y\in E$ such that $xy\notin\{x,y\}$. Now consider the maximal linked system $\mathcal L=\la\{x,y\},\{x,xy\},\{y,xy\}\ra$.
It can be shown that $\LL\ne \LL*\LL=\la\{xy\}\ra=\LL*\LL*\LL$,
which is not possible if the semigroup $\lambda(X)$ is
sub-Clifford.

Next, we show that the element $\LL$ is not regular in the
semigroup $\upsilon(X)$, which is not possible if the semigroup
$\lambda(X)$ is regular. Assuming that $\LL$ is regular, find an
upfamily $\A\in\upsilon(X)$ such that $\LL*\A*\LL=\LL$. It follows
from $\{x,y\}\in\LL=\LL*\A*\LL$ that $\{x,y\}\supset \bigcup_{u\in
L}u*B_u$ for some set $L\in\LL$ and some family $\{B_u\}_{u\in
L}\subset\A*\LL$. The linked property of family $\LL$ implies that
the intersection $L\cap\{x,xy\}$ contains some point $u$. Now for
the set $B_u\in\A*\LL$ find a set $A\in\A$ and a family
$\{L_a\}_{a\in A}\subset\LL$ such that $B_u\supset\bigcup_{a\in
A}a*L_a$. Fix any point $a\in A$ and a point $v\in
L_a\cap\{y,xy\}$. Then $uav\in uaL_a\subset uB_u\subset \{x,y\}$.
Since $u\in\{x,xy\}$ and $v\in \{y,xy\}$, the element $uav$ is
equal to $xby$ for some element $b\in\{a,ya,ax,yax\}$. So,
$xby\in\{x,y\}$. If $xby=x$, then $xy=xbyy=xby=x\in\{x,y\}$. If
$xby=y$, then $xy=xxby=xby=y\in\{x,y\}$. In both cases we obtain a
contradiction with the choice of the idempotents $x$ and $y$.

Since the idempotents of the semigroup $\beta(X)\subset\lambda(X)$
commute, the linear semilattice $E$ is finite according to
Proposition~\ref{p2.1}.
\end{proof}

Since $X$ is an inverse semigroup with finite linear semilattice
$E$, we can apply Theorem 7.5 of \cite{CP2} to derive our next claim.

\begin{claim}\label{cl6.4} The semigroup $X$ is inverse and Clifford.
\end{claim}

Since the semigroup $X$ is inverse and Clifford, the idempotents
of $X$ commute with all elements of $X$; see Theorem II.2.6 in
\cite{Pet}.

\begin{claim}\label{cl6.5} Each subgroup $H$ in $X$ has cardinality $|H|\le 4$.
\end{claim}

\begin{proof} We lose no generality assuming that $H$ coincides with the maximal group $H_e$ containing the idempotent $e$ of the group $H$. An upfamily $\A\in\upsilon(H)$ is called {\em left invariant} if $x\A=\A$ for any point $x\in H$. By $\inv[N]_2(H)$ denote the family of all left invariant linked systems on $H$ and by
$\inv[\lambda](H)=\max\inv[N]_2(H)$ the family of all maximal
elements of $\inv[N]_2(H)$.  Elements of $\inv[\lambda](H)$ are
called {\em maximal invariant linked systems}. Zorn's Lemma
guarantees that the set $\inv[\lambda](H)$ is not empty.

We claim that $\inv[\lambda](H)$ is a singleton. Assuming the
opposite, fix two distinct maximal invariant linked systems
$\A_1,\A_2\in\inv[\lambda](H)$. By Proposition 1 of \cite{BG3},
for every $i\in\{1,2\}$ the set
 $${\uparrow}\A_i=\{\LL\in\lambda(H):\LL\supset\A_i\}$$is a left ideal in the compact right-topological semigroup $\lambda(H)$. By Ellis'  Theorem \cite{Ellis} (see also \cite[2.5]{HS}), this left ideal contains an idempotent $\E_i\supset\A_i$. The idempotents $\E_1,\E_2$ do not commute because the products $\E_1*\E_2$ and $\E_2*\E_1$ belong to the disjoint left ideals ${\uparrow}\A_2$ and ${\uparrow}\A_1$, respectively (the left ideals ${\uparrow}\A_1$ and ${\uparrow}\A_2$ are disjoint by the maximality of the invariant linked systems $\A_1$ and $\A_2$).

Since $|\inv[\lambda](H)|=1$, we can apply Theorems 2.2 and 2.6 of
\cite{BGN} and conclude that the group $H$ either has cardinality
$|H|\le 5$ or else $H$ is isomorphic to the dihedral group $D_6$
or to the group $(C_2)^3$.

To finish the proof of Claim~\ref{cl6.5} it remains to show that
$H$ is not isomorphic to the groups $C_5$, $D_6$ or $C_2^3$.
\smallskip

$C_5$: If $H$ is isomorphic to the 5-element cyclic group $C_5$,
then the superextension $\lambda(X)$ contains an isomorphic copy
of the semigroup $\lambda(C_5)$. By \cite[\S6.4]{BGN}, the
semigroup $\lambda(H)\cong\lambda(C_5)$ contains two distinct
elements $\mathcal Z,\Theta$ such that $\LL*\Theta=\mathcal Z$ for
any maximal linked system $\LL\in\lambda(H)$. This implies that
the element $\Theta$ is not regular in $\lambda(H)$. We claim that
this element is not regular in $\lambda(X)$. Assuming the
converse, find a maximal linked system $\LL\in\lambda(X)$ such
that $\Theta=\Theta*\LL*\Theta$.

Let $e$ be the idempotent of the maximal subgroup $H=H_e$ of $X$.
Since $X$ is inverse and Clifford, the idempotent $e$ lies in the
center of the semigroup $X$, so the shift $s_e:X\to eX$,
$s_e:x\mapsto xe=ex$, is a well-defined homomorphism from the
semigroup $X$ onto its principal ideal $eX=Xe$. Since
$e\Theta=\Theta$, for the maximal linked system
$e\LL\in\lambda(eX)$ we get $\Theta=\Theta*e\LL*\Theta$, which
means that $\Theta$ is regular in the semigroup $\lambda(eX)$.
Since $eX\setminus H_e$ is an ideal in $eX$, Corollary~\ref{c3.4}
implies that the element $\Theta$ is regular in the semigroup
$\lambda(H_e)$, which contradicts the choice of $\Theta$. So,
$\Theta$ is not regular in $\lambda(X)$ and the semigroup
$\lambda(X)$ is not regular.

On the other hand, the property of $\Theta$ guarantees that
$\Theta\not=\mathcal Z=\Theta*\Theta=\Theta*\Theta*\Theta$, which
means that the semigroup $\lambda(X)$ is not sub-Clifford. In both
cases we obtain a contradiction with Claim~\ref{cl6.1}.
\smallskip

$D_6$: Next, assume that $H$ is isomorphic to the dihedral group
$D_6$. In this case $H$ contains an element $a$ of order 3 and
element $b$ of order 2 such that $ba=a^2b$. Consider the maximal
linked systems $\Delta=\langle\{e, a\}, \{e,a^2\}, \{a,
a^2\}\rangle$ and $\Lambda=\langle\{e, b\}, \{e, ab\}, \{e, a,
a^2\}, \{a, b,ab\}, \{a^2,b,ab\}\rangle$. It is easy to
check that $\Delta$ and $\Lambda$ are two non-commuting
idempotents in $\lambda(X)$ (because $\{e,a,ab\}\in\Delta*\Lambda$
and $\{e,a,ab\}\notin\Lambda*\Delta$). So, the semigroups
$\lambda(X)\supset\lambda(H)$ contains two non-commuting
idempotents, which is a contradiction.
\smallskip

$C_2^3$: In the case $H \cong C_2^3$ fix three elements $a,b,c\in
C_2^3$ generating the group $C_2^3$. Consider two maximal linked
systems $\square_b=\langle \{e,a\},\{e,b\},
\{e,ab\},\{a,b,ab\}\rangle$ and $\square_c=\langle
\{e,a\},\{e,c\},\{e,ac\},\{a,c,ac\}\rangle$ and observe that they
are non-commuting idempotents of $\lambda(H)$ (because
$\{e,c,b,ab\}\in \square_b*\square_c$ and $\{e,c,b,ab\}\notin
\square_c*\square_b$).
\end{proof}

\begin{claim}\label{cl6.6} The semigroup $\lambda(X)$ is inverse.
\end{claim}

\begin{proof} By Claim~\ref{cl6.1}, the semigroup $\lambda(X)$ is regular or sub-Clifford. We claim that $\lambda(X)$ is regular. If not, then $\lambda(X)$ is sub-Clifford. Claims~\ref{cl6.3}---\ref{cl6.5}  imply that the semigroup $X$ is finite and so is its superextension $\lambda(X)$. Being finite and sub-Clifford, the semigroup $\lambda(X)$ is Clifford and hence regular. Taking into account that the idempotents of $\lambda(X)$ commute, we conclude that the semigroup $\lambda(X)$ is inverse.
\end{proof}

Since the semilattice $E$ is linear, we can write it as
$E=\{e_1,\dots,e_n\}$ where $e_ie_j=e_i\ne e_j$ for all $1\le
i<j<n$. For every $i\le n$ by $H_i=H_{e_i}$ denote the maximal
subgroup of $X$ that contains the idempotent $e_i$. By
Claim~\ref{cl6.5}, each subgroup $H_i$ has cardinality $|H_i|\le
4$ and hence is commutative. We claim that the semigroup $X$ also
is commutative. Indeed, given any points $x,y\in X$ we can find
numbers $i,j\le n$ such that $x\in H_i$ and $y\in H_j$. We lose no
generality assuming that $i\le j$. Then
$xy=(xe_i)y=x(e_iy)=(ye_i)x=yx\in H_i$, so $X$ is commutative.

\begin{claim}\label{cl6.7} For any $1<i<n$ the maximal subgroup $H_i$ is trivial.
\end{claim}

\begin{proof} Assume conversely that the subgroup $H_i$ is not trivial and take any element $a\in H_i\setminus E$.
Next, consider the maximal linked system
$\Delta=\la\{e_{i-1},a\},\{a,e_{i+1}\},\{e_{i-1},e_{i+1}\}\ra$. We
claim that $\Delta$ is not regular in $\lambda(X)$. Assume
conversely that $\Delta$ is regular in $\lambda(X)$. Using
Corollary~\ref{c3.4}, we can show that $\Delta$ is regular in the
semigroup $\lambda(H_{i-1}\cup H_i\cup H_{i+1})$. Then we can find
a maximal linked system $\F\in\lambda(H_{i-1}\cup H_i\cup
H_{i+1})$ such that $\Delta=\Delta*\F*\Delta$. For the set
$\{a,e_{i+1}\}\in\Delta$, find a set $A\in\Delta*\F$ with
$A\subset H_{i-1}\cup H_i\cup H_{i+1}$ and a family $\{D_a\}_{a\in
A}\subset\Delta$ such that $\{a,e_{i+1}\}\supset\bigcup_{a\in
A}a*D_a$. Observe that such an inclusion is possible only if
$D_a=\{a,e_{i+1}\}$ for all $a\in A$. But then
$A*\{a,e_{i+1}\}\subset\{a,e_{i+1}\}$ implies $A=\{e_{i+1}\}$ and
$\Delta*\F=\la e_{i+1}\ra$ which is not possible.
\end{proof}

\begin{claim}\label{cl6.8} If $n\ge 2$, then $|H_n|\le 2$.
\end{claim}

\begin{proof} Assume conversely that $|H_n|>2$. Then two cases are possible.
\smallskip

1. The group $H_n$ is cyclic. Fix a generator $a$ of the cyclic
group $H_n$ and consider the maximal linked system $\Delta=\la
\{a,e_{n-1}\},\{a,e_n\},\{e_{n-1},e_n\}\ra$. We claim that
the element $\Delta$ is not regular in the semigroup $\lambda(X)$.
Assuming the opposite, we can find a maximal linked system
$\F\in\lambda(X)$ with $\Delta*\F*\Delta=\Delta$. Then for the set
$\{e_n,a\}\in\Delta$ we can find a set $A\in\Delta*\F$ and a
family $\{D_a\}_{a\in A}$ of minimal subsets of $\Delta$ such that
$\{e_n,a\}\supset \bigcup_{a\in A}a*D_a$. This inclusion is
possible only if $\{a,e_n\}=D_a$ for all $a\in A$. The inclusion
$A*\{a,e_n\}\subset\{a,e_n\}$ implies that $A=\{e_n\}$. Now for
the set $A\in\Delta*\F$, find a minimal set $D\in\Delta$ and a
family $\{F_d\}_{d\in D}\subset\F$ such that $\bigcup_{d\in
D}d*F_d\subset A=\{e_n\}$. This inclusion is possible only if
$\{a,e_n\}\subset D\subset H_n$ and $A_d=\{d^{-1}\}\subset H_n$
for each $d\in D$. Then the family $\A$ contains two disjoint sets
$\{e_n\}$ and $\{a^{-1}\}$ which is not possible as $\A$ is
linked.
\smallskip

2. The group $H_n$ is isomorphic to the group $C_2\times C_2$.
Then we can take two distinct elements $a,b$ generating the group
$H_n$, and consider the maximal linked system
$\square=\la\{e_{n-1},a\},\{e_{n-1},b\},\{e_{n-1},ab\},\{a,b,ab\}\ra$.
We claim that the element $\square$ is not regular in the
semigroup $\lambda(X)$. Assuming the opposite, find a maximal
linked system $\F\in\lambda(X)$ with $\square*\F*\square=\square$.
Then for the set $\{a,b,ab\}\in\square$ we can find a set
$A\in\square*\F$ and a family $\{D_x\}_{x\in A}$ of minimal
subsets of $\square$ such that $\{a,b,ab\}\subset \bigcup_{x\in
A}x*D_x$. This inclusion is possible only if $D_x=\{a,b,ab\}$ for
all $x\in A$. The inclusion $ A*\{a,b,ab\}=\bigcup_{x\in
A}x*D_x\subset\{a,b,ab\}$ implies that $A=\{e_n\}$. Now for the
set $A\in\square*\F$, find a minimal set $S\in\Box$ and a family
$\{F_s\}_{s\in S}\subset\F$ such that $\bigcup_{s\in
S}s*F_s\subset A=\{e_n\}$. This inclusion is possible only if
$S=\{a,b,ab\}$ and $F_s=\{s\}\subset H_n$ for each $d\in D$. Then
the family $\F$ contains disjoint sets $\{s\}$, $s\in S$, which is
not possible as $\F$ is linked.
\end{proof}

\begin{claim}\label{cl6.9} If $n\ge 3$, then the group $H_n$ is trivial.
\end{claim}

\begin{proof} Assume that $H_n$ is not trivial. By Claim~\ref{cl6.8}, $|H_n|=2$. Fix a generator $a$ of the cyclic group $H_n$ and observe that the maximal linked systems  $\square=\la\{e_{n-2},a\},\{e_{n-2},e_{n-1}\},\{e_{n-2},e_n\},\{a,e_{n-1},e_n\}\ra$ and
$\Delta=\la\{e_{n-1},a\},\{e_{n-1},e_n\},\{a,e_n\}\ra$ are
non-commuting idempotents of the se\-migroup $\lambda(X)$ because
$$\square*\Delta=\la\{e_{n-1},e_{n-2}\},\{e_{n-1},a\},\{e_{n-1},e_n\},\{e_{n-2},a,e_n\}\ra\neq\square=
\Delta*\square.$$
\end{proof}

\begin{claim}\label{cl6.10} If $n\ge 2$, then $|H_1|\le 2$.
\end{claim}

\begin{proof}

Assume that $|H_1|>2$ and chose two distinct elements $a,b\in
H_1\setminus E$. We claim that the maximal linked system
$$\Delta=\big\la\{a,b\},\{a,e_2\},\{b,e_2\}\big\ra$$ is
not regular element of $\lambda(X)$. Assuming the opposite, find a
maximal linked system $\F\in\lambda(X)$ such that
$\Delta*\F*\Delta=\Delta$. Replacing $\F$ by $e_2*\F$, if
necessary, we can assume that $\F\in\lambda(H_1\cup H_2)$. For the
set $\{a,e_2\}\in\Delta$, find a set $A\in\Delta*\F$ and a family
$\{D_x\}_{x\in A}\subset\Delta$ such that $\bigcup_{x\in
A}x*D_x\subset\{a,e_2\}$. This inclusion implies that
$A=\{e_2\}\in\Delta*\F$, which is not possible.
\end{proof}

Now we are able to finish the proof of Theorem~\ref{t1l}. If
$n=|E|\ge 3$, then the semigroup $X$ is isomorphic to $L_n$ or to
$C_2\sqcup L_{n-1}$ by Claims~\ref{cl6.7}---\ref{cl6.10}. If
$n=|E|=1$, then $X$ is a group of cardinality $|X|\le 4$,
isomorphic to one of groups: $L_1$, $C_2$, $C_3$, $C_4$,
$C_2\times C_2$.

It remains to consider the case $|E|=2$. By Claims~\ref{cl6.8},
\ref{cl6.10}, $\max\{|H_1|,|H_2|\}\le 2$. If $|H_1|=|H_2|=1$,
then $X\cong L_2$. If $|H_1|=1$ and $|H_2|=2$, then $X\cong
L_1\sqcup C_2$. If $|H_1|=2$ and $|H_2|=1$, then $X\cong C_2\sqcup
L_1$.

Finally assume that $|H_1|=|H_2|$. For $i\in\{1,2\}$ let $a_i$ be
the unique generator of the 2-element cyclic group $H_i$.

\begin{claim}\label{cl6.11} $a_2*e_1=a_1$.
\end{claim}

\begin{proof} Assuming that $a_2*e_1\ne a_1$, we get $a_2*e_1=e_1$. Then the maximal linked systems
$$\square_e=\la\{e_1,a_1\},\{e_1,a_2\},\{e_1,e_2\},\{a_1,a_2,e_2\}\ra\mbox{ \ and \ }
$$ $$\square_a=\la\{a_1,e_1\},\{a_1,e_2\},\{a_1,a_2\},\{e_1,e_2,a_2\}\ra$$
are not commuting idempotents of $\lambda(X)$ (because
$\square_e*\square_a=\square_a$ while
$\square_a*\square_e=\square_e$).
\end{proof}

Claim~\ref{cl6.11} implies that the semigroup $X=H_1\cup H_2$ is
isomorphic to $L_2\times C_2$.

\section{Proof of Theorem~\ref{t1f}}\label{s:t1f}

Given a semigroup $X$, we need to check the equivalence of the
following statements:
\begin{enumerate}
\item[(1)] $\varphi(X)$ is a commutative Clifford semigroup;
\item[(2)] $\varphi(X)$ is an inverse semigroup;
\item[(3)] the idempotents of
$\varphi(X)$ commute and $\varphi(X)$ is sub-Clifford or regular
in $N_2(X)$;
\item[(4)] $X$ is isomorphic to $C_2$, $L_n$, or
$L_n\sqcup C_2$ for some $n\in\w$.
\end{enumerate}

We shall prove the implications $(4)\Ra(1)\Ra(2)\Ra(3)\Ra(4)$.

The implication $(4)\Ra(1)$ follows from Propositions~\ref{p4.7},
\ref{p4.8} while $(1)\Ra(2)\Ra(3)$ are trivial or well-known; see
\cite[II.1.2]{Pet}.

To prove that $(3)\Ra(4)$, assume that idempotents of the
semigroup $\varphi(X)$ commute and $\varphi(X)$ is sub-Clifford or
regular in $N_2(X)$. Then the idempotents of the semigroup $X$
commute and thus the set $E=\{e\in X:ee=e\}$ is a commutative
subsemigroup of $X$. By analogy with Claim~\ref{cl6.2} we can
prove:

\begin{claim} The semigroup $X$ is inverse.
\end{claim}

\begin{claim} The semilattice $E$ is linear and finite.
\end{claim}

\begin{proof} Assuming that the semilattice $E$ is not linear, we can find two non-commuting idempotents $x,y\in E$. By (the proof of) Theorem 1.1 of \cite{BGs}, the filter $\F=\big\la\{x,y\}\big\ra$ is not a regular element in $\upsilon(X)$ and $\F\ne \la\{x,y,xy\}\ra=\F*\F=\F*\F*\F$, which is not possible if the semigroup $\varphi(X)$ is  sub-Clifford or regular in $N_2(X)$. So, the semilattice $E$ is linear. Since $\beta(X)\subset\varphi(X)$, Proposition~\ref{p2.1} implies that the linear semilattice $E$ is finite.
\end{proof}

Since $X$ is an inverse semigroup with finite linear semilattice
$E$, we can apply Theorem 7.5 of \cite{CP2} to derive our next claim.

\begin{claim} The semigroup $X$ is inverse and Clifford.
\end{claim}

Since the semigroup $X$ is inverse and Clifford, the idempotents
of $X$ commute with all elements of $X$; see \cite[II.2.6]{Pet}.

\begin{claim}\label{cl8.4} Each subgroup $H$ in $X$ has cardinality $|H|\le 2$.
\end{claim}

\begin{proof} Assume $X$ contains a subgroup $H$ of cardinality $|H|>2$. We lose no generality assuming that the subgroup $H$ coincides with the maximal subgroup $H_e$ containing the idempotent $e$ of $H$. Take any subset $F\subset H$ with $|H\setminus F|=1$ and consider the filter $\F=\la F\ra$. It follows that $\F\ne \la H\ra=\F*\F=\F*\F*\F$, which is forbidden if the semigroup $\lambda(X)$ is sub-Clifford.
Next, we show that $\F$ is not regular in the semigroup
$N_2(X)\supset\varphi(X)$. Assuming the opposite, find an upfamily
$\A\in N_2(X)$ such that $\F=\F*\A*\F$. Replacing $\A$ by the
linked upfamily $e\A$, we can assume that $\A\in N_2(eX)$. Since
$eX\setminus H_e$ is an ideal in $eX$, Corollary~\ref{c3.4}
implies that the $\F$ is a regular element of the semigroup
$N_2(H)$ and hence we can assume that $\A\in N_2(H)$. Then
$\F*\A\in N_2(H)\subset N_2(X)$. The inclusion $F\in\F=\F*\A*\F$
implies the existence of a set $B\in\F*\A$, $B\subset H$, and a
family $\{F_b\}_{b\in B}\subset \F$ such that $\bigcup_{b\in
B}b*F_b\subset F$. Replacing $F_b$ by the smallest possible set
$F$ generating the filter $\F$, we can assume that $F_b=F$ for all
$b\in B$. Then we get $B*F=\bigcup_{b\in B}b*F_b\subset F$ and
hence $B=\{e\}$. Since $B\in\F*\A$, for the smallest set $F\in\F$
and each point $x\in F$ we can find a set $A_x\subset\A$,
$A_x\subset H$, such that $\bigcup_{x\in F}x*A_x\subset\{e\}$. It
follows that $A_x=\{x^{-1}\}$ and hence the family
$\A\supset\{A_x:x\in F\}$ is not linked, which is a desired
contradiction. So, $\F$ is not regular in $N_2(X)$ and
$\F\ne\F*\F=\F*\F*\F$, which is not possible if the semigroup
$\lambda(X)$ is sub-Clifford or regular in $N_2(X)$.
\end{proof}

By analogy with Claim~\ref{cl6.6} we can prove:

\begin{claim}\label{cl8.5} The semigroup $\varphi(X)$ is regular in $N_2(X)$.
\end{claim}

Since the semilattice $E$ is linear, we can write it as
$E=\{e_1,\dots,e_n\}$ where $e_ie_j=e_i\ne e_j$ for all $1\le
i<j\le n$. For every $i\le n$ by $H_i=H_{e_i}$ denote the maximal
subgroup of $X$ that contains the idempotent $e_i$. By
Claim~\ref{cl8.4}, each subgroup $H_i$ has cardinality $|H_i|\le
2$ and hence is commutative. Then the inverse Clifford semigroup
$X$ also is commutative.

\begin{claim}\label{cl8.6} For any $1\le i<n$ the maximal subgroup $H_i$ is trivial.
\end{claim}

\begin{proof} Assume conversely that for some $i<n$ the subgroup $H_i$ is not trivial and take any element $a\in H_i\setminus E$.
Next, consider the filter $\F=\la F\ra$ generated by the
doubleton $F=\{a,e_{i+1}\}$. We claim that $\F$ is a non-regular
element in the semigroup $N_2(X)$, which will contradict
Claim~\ref{cl8.5}.

Assuming that $\F$ is regular in $N_2(X)$, we can find a linked
upfamily $\A\in N_2(X)$ such that $\F=\F*\A*\F$. Replacing $\A$ by
$e_{i+1}\A$, we can assume that $\A\in N_2(e_{i+1}X)$.
 For the set $F=\{a,e_{i+1}\}\in\F$, find a set $B\in\F*\A$ such that $B*F\subset F$.
The latter inclusion implies that $B\subset H_i\cup H_{i+1}$. The
inclusion $B\in \F*\A$ implies that the intersection $B\cap H_i$
is not empty and the inclusion $B*F\subset F$ implies that $B\cap
H_i=\{e_i\}$ and then $\{a,e_{i+1}\}=F\supset B*F\supset
\{e_i\}*\{a,e_{i+1}\}=\{a,e_i\}$, which is a desired
contradiction.
\end{proof}

Now we are able to finish the proof of Theorem~\ref{t1f}. By
Claim~\ref{cl8.6}, all maximal subgroups $H_i$, $i<n$, are
trivial. If the group $H_n$ is trivial, then $X=E$ is isomorphic
to the linear semilattice $L_n$. If $H_n$ is not trivial, then
$H_n\cong C_2$ by Claim~\ref{cl8.4} and $X$ is isomorphic to the
semigroup $L_{n-1}\sqcup C_2$. For $n=1$ we get $L_0=\emptyset$
and $L_0\sqcup C_2=C_2$.

\section{Proof of Theorem~\ref{t1n}}\label{s:t1n}

Given a semigroup $X$ we need to prove the equivalence of the
following statements:
\begin{enumerate}
\item[(1)] $N_2(X)$ is a finite commutative Clifford semigroup;
\item[(2)] $N_2(X)$ is an inverse semigroup;
\item[(3)] idempotents of $N_2(X)$
commute and $N_2(X)$ is sub-Clifford or regular;
\item[(4)] $X$ is isomorphic to  $C_2$ or $L_n$ for some $n\in\w$.
\end{enumerate}

We shall prove the implications $(4)\Ra(1)\Ra(2)\Ra(3)\Ra(4)$.
\smallskip

The implication $(4)\Ra(1)$ follows from
Propositions~\ref{p4.7}(1) and \ref{p4.2} while $(1)\Ra(2)\Ra(3)$
are trivial or well-known; see \cite[II.1.2]{Pet}.
\smallskip

To prove that $(3)\Ra(4)$, assume that idempotents of the
semigroup $N_2(X)$ commute and $N_2(X)$ is sub-Clifford or
regular. Then the idempotents of the subsemigroups $\lambda(X)$
and $\varphi(X)$ of $N_2(X)$ also commute  and these semigroups
are regular in the semigroup $N_2(X)$. By Theorems~\ref{t1l} and
\ref{t1f}, the semigroup $X$ is isomorphic to one of the
semigroups $C_2$, $L_1\sqcup C_2$ or $L_n$ for some $n\in\w$. It
remains to prove that $X$ cannot be isomorphic to $L_1\sqcup
C_2=\{0,1,-1\}$.

This follows from the fact that the semigroup $N_2(\{0,1,-1\})$
contains two idempotents $\Delta=\{A\subset \{0,1,-1\}:|A|\ge 2\}$
and $\F=\la\{\{0,1,-1\}\ra$, which do not commute because
$\Delta*\F=\F\ne\la \{0,1\},\{0,-1\}\ra=\F*\Delta$.

\section{Proof of Theorem~\ref{t1u}}\label{s:t1u}

Given a semigroup $X$ we need to prove the equivalence of the
following statements:
\begin{enumerate}
\item[(1)] $\upsilon(X)$ is a finite commutative Clifford semigroup;
\item[(2)] $\upsilon(X)$ is an inverse semigroup;
\item[(3)] idempotents of
$\upsilon(X)$ commute and $\upsilon(X)$ is sub-Clifford or
regular;
\item[(4)] $X$ is a finite linear semilattice, isomorphic to $L_n$ for some $n\in\w$.
\end{enumerate}
\smallskip

We shall prove the implications $(4)\Ra(1)\Ra(2)\Ra(3)\Ra(4)$.
\smallskip

The implication $(4)\Ra(1)$ follows from Proposition~\ref{p4.2}
while $(1)\Ra(2)\Ra(3)$ are trivial.
\smallskip

To prove that $(3)\Ra(4)$, assume that the idempotents of the
semigroup $\upsilon(X)$ commute and $\upsilon(X)$ is sub-Clifford
or regular. Then the idempotents of the semigroup $X$ commute and
thus the set $E=\{e\in X:ee=e\}$ is a commutative subsemigroup of
$X$. By analogy with Claims~\ref{cl6.2}---\ref{cl6.4} we can prove
that the semigroup $X$ is inverse and Clifford and the semilattice
$E$ is finite and linear.

Next, we show that each subgroup $H$ of $X$ is trivial.  Assume
conversely that $X$ contains a non-trivial subgroup $H$. Then the
filter $\F=\la H\ra$ and the upfamily $\U=\{A\subset
H:A\ne\emptyset\}$ are two non-commuting idempotents in the
semigroup $\upsilon(H)\subset\upsilon(X)$ (because
$\F*\U=\U\ne\F=\U*\F$).

Now we see that the inverse Clifford semigroup $X$ contains no
non-trivial subgroups and hence coincides with its maximal
semilattice $E$, which is finite and linear.


\begin{thebibliography}{10}


\bibitem{BG2} T.~Banakh, V.~Gavrylkiv, {\em Algebra in superextension of groups, II:
cancelativity and centers}, Algebra Discr. Math. (2008), No.4, 1--14.

\bibitem{BG3} T.~Banakh, V.~Gavrylkiv, {\em Algebra in superextension of groups:
minimal left ideals}, Mat. Stud. {\bf 31} (2009), 142--148.

\bibitem{BG10} T.~Banakh, V.~Gavrylkiv, {\em Algebra in the superextensions of twinic groups}, Dissert. Math. {\bf 473} (2010), 74pp.

\bibitem{BGs} T.~Banakh, V.~Gavrylkiv, {\em Algebra in superextensions of
semilattices}, Algebra Discr. Math. {\bf 13} (2012), No.1, 26--42.

\bibitem{BGN} T.~Banakh, V.~Gavrylkiv, O.~Nykyforchyn, {\em Algebra
in superextensions of groups, I: zeros and commutativity},
Algebra Discr. Math. (2008), No.3, 1--29.



\bibitem{CP2} A.H.~Clifford, G.B.~Preston, { The algebraic theory of semigroups}. Vol. II.,
Mathematical Surveys. {\bf 7}. AMS, Providence, RI, 1967.



\bibitem{Ellis} R.~Ellis, {\em Distal transformation groups}, Pacific J.~Math. {\bf 8} (1958), 401--405.

\bibitem{G1} V.~Gavrylkiv, {\em The spaces of inclusion hyperspaces over noncompact
spaces}, Mat. Stud. {\bf 28}:1 (2007), 92--110.

\bibitem{G2} V.~Gavrylkiv, {\em Right-topological semigroup operations on inclusion hyperspaces},
 Mat. Stud. {\bf 29}:1 (2008), 18--34.

\bibitem{GRS} R.~Graham, B.~Rothschild, J.~Spencer, {Ramsey theory}, John Wiley \&\ Sons, Inc., New York, 1990.



\bibitem{HS} N.~Hindman, D.~Strauss, {Algebra in the Stone-\v Cech compactification}, de Gruyter, Berlin, New York, 1998.

\bibitem{vM} J.~van Mill,
Supercompactness and Wallman spaces, Math. Centre Tracts. {\bf
85}. Amsterdam: Math. Centrum., 1977.

\bibitem{Pet} M.~Petrich, {Inverse semigroups}, John Wiley \&\ Sons, Inc., New York, 1984.


\bibitem{Ramsey} F.~Ramsey, {\em On a problem of formal logic}, Proc. London Math. Soc. {\bf 30} (1930), 264--286.



\bibitem{Ve} A.~Verbeek, Superextensions of topological spaces, MC
Tract 41, Amsterdam, 1972.
\end{thebibliography}
\end{document}